\documentclass[11pt, reqno, oneside]{amsart}
\usepackage{amsmath,amsfonts,amssymb,amsthm}

\usepackage{enumerate}
\usepackage{esint}
\usepackage[pagebackref,colorlinks,citecolor=blue,linkcolor=blue]{hyperref}
\usepackage[dvipsnames]{xcolor}
\usepackage{mathtools}  

\usepackage{mathtools}
\usepackage{setspace}

\usepackage{comment}

\usepackage{geometry}
\numberwithin{equation}{section}

\theoremstyle{plain}
\newtheorem{theorem}[equation]{Theorem}
\newtheorem{corollary}[equation]{Corollary}
\newtheorem{lemma}[equation]{Lemma}
\newtheorem{proposition}[equation]{Proposition}

\theoremstyle{definition}
\newtheorem{definition}[equation]{Definition}

\newtheorem{example}[equation]{Example}

\newcommand{\R}{{\mathbb R}}
\newcommand{\N}{{\mathbb N}}

\newcommand{\Om}{\Omega}

\providecommand{\vint}[1]{\mathchoice
	{\mathop{\vrule width 5pt height 3 pt depth -2.5pt
			\kern -9pt \kern 1pt\intop}\nolimits_{\kern -5pt{#1}}}
	{\mathop{\vrule width 5pt height 3 pt depth -2.6pt
			\kern -6pt \intop}\nolimits_{\kern -3pt{#1}}}
	{\mathop{\vrule width 5pt height 3 pt depth -2.6pt
			\kern -6pt \intop}\nolimits_{\kern -3pt{#1}}}
	{\mathop{\vrule width 5pt height 3 pt depth -2.6pt
			\kern -6pt \intop}\nolimits_{\kern -3pt{#1}}}}

\newcommand{\eps}{\varepsilon}
\newcommand{\loc}{\mathrm{loc}}

\newcommand{\BV}{\mathrm{BV}}

\newcommand{\ch}{\text{\raise 1.3pt \hbox{$\chi$}\kern-0.2pt}}

\DeclareMathOperator{\capa}{Cap}
\DeclareMathOperator{\rcapa}{cap}
\DeclareMathOperator{\dive}{div}
\DeclareMathOperator{\dist}{dist}
\DeclareMathOperator{\diam}{diam}

\DeclareMathOperator{\Lip}{Lip}

\DeclareMathOperator{\Var}{Var}
\DeclareMathOperator{\pV}{pV}

\DeclareMathOperator{\fineint}{fine-int}

\begin{document}
\title[]{Finely quasiconformal mappings
}
\author{Panu Lahti}
\address{Panu Lahti,  Academy of Mathematics and Systems Science, Chinese Academy of Sciences,
	Beijing 100190, PR China, {\tt panulahti@amss.ac.cn}}


\begin{abstract}
	We introduce a relaxed version of the metric definition of quasiconformality that
	is natural also for mappings of low regularity, including $W_{\loc}^{1,1}(\R^n;\R^n)$-mappings.
	Then we show on the plane that this relaxed definition can be used to prove Sobolev regularity,
	and that these ``finely quasiconformal'' mappings are in fact quasiconformal.
\end{abstract}

\date{\today}
\maketitle

%

\section{Introduction}

A homeomorphism $f\colon \R^n\to\R^n$ is said to be quasiconformal if
$K_f(x)\le K<\infty$ for all $x\in\R^n$, where
\[
K_f(x):=\limsup_{r\to 0}\left(\frac{\diam f(B(x,r))^n}{|f(B(x,r))|}\right)^{1/(n-1)}.
\]
We always consider $n\ge 2$, and we use $|\cdot|$ for the Euclidean norm as well as for the
Lebesgue measure. There are several equivalent definitions of quasiconformality;
the above is a ``metric'' definition.
As part of an ``analytic definition'',
it is known that quasiconformal mappings are in the Sobolev class
$W_{\loc}^{1,n}(\R^n;\R^n)$.

There has been wide interest in showing that if quasiconformality is assumed in some relaxed sense,
it follows that the mapping in question is in fact quasiconformal,
or at least has some lower regularity, such as $W^{1,1}_{\loc}$-regularity.
For example, Koskela--Rogovin \cite[Corollary 1.3]{KoRo} show that if
$f\colon \R^n\to\R^n$ is a homeomorphism,
$K_f\in L^1_{\loc}(\R^n)$, and $K_f<\infty$ outside a set of
$\sigma$-finite $\mathcal H^{n-1}$-measure, then $f\in W_{\loc}^{1,1}(\R^n;\R^n)$.
Many results in the same vein have been proven starting from
Gehring \cite{Ge62,Ge}, see also
Balogh--Koskela \cite{BaKo},
Fang \cite{Fan},
Heinonen--Koskela--Shanmugalingam--Tyson \cite{HKST},
Kallunki--Koskela \cite{KaKo},
and
Margulis--Mostow \cite{MaMo}.
Several works study specifically the issue of $W^{1,1}_{\loc}$-regularity, see
Balogh--Koskela--Rogovin \cite{BKR},
Kallunki--Martio \cite{KaMa},
and Williams \cite{Wi}.

The quantity
$K_f^{n-1}$ can be essentially thought of as ``$|\nabla f|^n$ divided by the Jacobian determinant''.
Indeed, for a quasiconformal mapping $f\colon \R^n\to\R^n$, we know that
\begin{equation}\label{eq:Kf n minus one}
	K_f(x)^{n-1}|\det \nabla f(x)|
	= \frac{2^n}{\omega_n}\Vert \nabla f(x)\Vert^n
	\quad \textrm{for a.e. }x\in \R^n,
\end{equation}
where $\omega_n$ is the Lebesgue
measure of the unit ball,
$\Vert \cdot\Vert$ is the maximum norm,
and $\nabla f$ can be understood to be either the classical gradient
or the weak gradient.
With the latter interpretation, all of the quantities in \eqref{eq:Kf n minus one}
make sense also for mappings of lower Sobolev regularity, but the equality
can fail already for $W_{\loc}^{1,n}$-mappings
--- let alone $W_{\loc}^{1,1}$-mappings --- since for them $\diam f(B(x,r))$
can easily be $\infty$ for every $x\in \R^n$ and $r>0$; see Example \ref{ex:W11 function}.
The problem is that the quantity $K_f$ is very
sensitive to oscillations, and essentially tailored to mappings $f$ that have better than
$W_{\loc}^{1,n}(\R^n;\R^n)$-regularity.
We wish to find a quantity that corresponds to
``$|\nabla f|^n$ divided by the Jacobian determinant'' in the case of $W^{1,1}_{\loc}$-mappings.
Hence we define the relaxed quantities
\[
K_{f,U}(x,r):=\left(\frac{\diam f(B(x,r)\cap U)^n}{|f(B(x,r))|}\right)^{1/(n-1)}
\quad\textrm{and}\quad
K_f^{\mathrm{fine}}(x):=\inf\limsup_{r\to 0}K_{f,U}(x,r),
\]
where the infimum is taken over \emph{1-finely open} sets $U\ni x$;
we give definitions in Section \ref{sec:prelis}.
In the following analog of \eqref{eq:Kf n minus one}, $f^*$ is the so-called precise representative of $f$.

\begin{theorem}\label{thm:K fine basic}
For every $f\in W_{\loc}^{1,1}(\R^n;\R^n)$, we have
\[
K_{f^*}^{\mathrm{fine}}(x)^{n-1}|\det \nabla f(x)|
\le \frac{2^n}{\omega_n} \Vert \nabla f(x)\Vert^n
\quad \textrm{for a.e. }x\in \R^n,
\]
with the interpretation $\infty\times 0=0$ if $\det \nabla f(x)=0$.
\end{theorem}

This shows that $K_{f}^{\mathrm{fine}}$ is generally much smaller than $K_f$.
On the other hand, the mapping we give in the aforementioned Example \ref{ex:W11 function}
is by no means a homeomorphism.
Thus one can ask: for a homeomorphism $f$,
are conditions on $K_f^{\mathrm{fine}}$ enough to prove Sobolev regularity,
or even quasiconformality?
Our main result is the following analog on the plane of the aforementioned Koskela--Rogovin \cite{KoRo},
and of other similar results.

\begin{theorem}\label{thm:main}
	Let $f\colon \R^2\to \R^2$ be a homeomorphism. Let $1\le p\le 2$.
	Suppose $K_f^{\mathrm{fine}}\in L^{p^*/2}_{\loc}(\R^2)$
	and $K_f^{\mathrm{fine}}<\infty$ outside a set $E$ of
	$\sigma$-finite $\mathcal H^{1}$-measure. Then $f\in W_{\loc}^{1,p}(\R^2;\R^2)$,
	and in the case $p=2$  we obtain that $f$ is quasiconformal and that
	$K_f^{\mathrm{fine}}(x)=K_f(x)$ for a.e. $x\in\R^2$. 
\end{theorem}

Here $p^*=2p/(2-p)$ when $1\le p<2$, and $p^*=\infty$ when $p=2$.
In the case $1\le p<2$, this theorem can be viewed as a statement about
``finely quasiconformal'' mappings of low regularity.
In the case $p=2$, we get the following corollary saying that ``finely quasiconformal'' mappings
are in fact quasiconformal.

\begin{corollary}\label{cor:main}
		Let $f\colon \R^2\to \R^2$ be a homeomorphism and suppose that
		$K_f^{\mathrm{fine}}(x)\le K<\infty$ for every $x\in \R^2$. Then
	$f$ is quasiconformal.
\end{corollary}

\section{Preliminaries}\label{sec:prelis}

Our definitions and notation are standard, and
the reader may consult e.g. the monograph Evans--Gariepy \cite{EvGa} for more background.
We will work in the Euclidean space $\R^n$ with $n\ge 2$.
We denote the $n$-dimensional Lebesgue outer measure by $\mathcal L^n$.
We denote the $s$-dimensional Hausdorff content by $\mathcal H_{R}^{s}$
and the Hausdorff measure by $\mathcal H^{s}$, with $0<R\le \infty$ and $0\le s\le n$.
If a property holds outside a set of Lebesgue measure zero, we say that it holds almost everywhere,
or ``a.e.''.

We denote the characteristic function of a set $A\subset\R^n$ by $\ch_A\colon \R^n\to \{0,1\}$.
We denote by $|v|$ the Euclidean norm of $v\in \R^n$,
and we also write $|A|:=\mathcal L^n(A)$ for a set $A\subset \R^n$.
We write $B(x,r)$ for an open ball in $\R^n$ with center $x\in \R^n$
and radius $r>0$, that is, $B(x,r)=\{y \in \R^n \colon |y-x|<r\}$.
We will sometimes use the notation $2B(x,r):=B(x,2r)$.
For matrices $A\in \R^{n\times n}$, we consider the Euclidean norm $|A|$ as well as the maximum norm
\[
\Vert A\Vert:=\max_{v\in\R^n,\,|v|=1}|Av|.
\]

By ``measurable'' we mean $\mathcal L^n$-measurable, unless otherwise specified.
If a function $u$ is in $L^1(D)$ for some measurable set $D \subset \R^n$ of nonzero and finite Lebesgue
measure, we write
\[
u_D:=\vint{D} u(y) \,d\mathcal L^n(y) \coloneqq \frac{1}{\mathcal L^n(D)} \int_D u(y) \,d\mathcal L^n(y)
\] 
for its mean value in $D$.

We will always denote by $\Om\subset\R^n$ an open set, and we consider $1\le p <\infty$.
Let $l\in\N$.
The Sobolev space $W^{1,p}(\Om;\R^l)$ consists of mappings $f\in L^p(\Om;\R^l)$
whose first weak partial derivatives $\partial f_j/\partial x_k$,
$j=1,\ldots,l$, $k=1,\ldots,n$, belong to $L^p(\Om)$.
We will only consider $l=1$ or $l=n$.
The weak partial derivatives form the matrix $(\nabla f)_{jk}$.
The Dirichlet space $D^{p}(\Om;\R^l)$ is defined in the same way, except that the integrability requirement
for the mapping itself is relaxed to $f\in L_{\loc}^1(\Om;\R^l)$.
The Sobolev norm is
\[
\Vert f\Vert_{W^{1,p}(\Om;\R^l)}
:=\Vert f\Vert_{L^p(\Om;\R^l)}+\Vert \nabla f\Vert_{L^p(\Om;\R^{l\times n})},
\]
where the $L^p$ norms are defined with respect to the Euclidean norm.

Consider a homeomorphism $f\colon \Om\to \Om'$, with $\Om,\Om'\subset \R^n$ open.
In addition to the Jacobian determinant $\det \nabla f(x)$, we also define the Jacobian
\begin{equation}\label{eq:Jacobian}
	J_f(x):=\limsup_{r\to 0}\frac{|f(B(x,r))|}{|B(x,r)|},\quad x\in\Om.
\end{equation}
Note that $J_f$ is the density of the pullback measure
\[
f_{\#}\mathcal L^n(A):=\mathcal L^n(f(A)),\quad \textrm{for Borel }A\subset \Om.
\]
By well-known results on densities, see e.g. \cite[p. 42]{EvGa}, we know the following:
$J_f(x)$ exists as a limit for a.e. $x\in\Om$, is a Borel function, and
\begin{equation}\label{eq:Jacobian basic}
	\int_{\Om}J_f\,d\mathcal L^n\le |f(\Om)|.
\end{equation}
Equality holds if $f$ is absolutely continuous in measure, that is,
if $|A|=0$ implies $|f(A)|=0$.
We will use the following ``analytic'' definition of quasiconformality.
For the equivalence of different definitions of quasiconformality, including the metric definition
used in the introduction, see e.g. \cite[Theorem 9.8]{HKST}.

\begin{definition}\label{def:quasiconformal}
Let $\Om,\Om'\subset \R^n$ be open sets.
A homeomorphism $f\in W_{\loc}^{1,n}(\Om;\Om')$ is said to be quasiconformal if
\begin{equation}\label{eq:Df and J 2}
	\Vert \nabla f(x)\Vert_{\max}^n\le K|\det \nabla f(x)|
	\quad\textrm{for a.e. }x\in \Om,
\end{equation}
for some constant $K<\infty$.
\end{definition}

Here we understand $\nabla f$ to be the weak gradient. However,
as a homeomorphism, $f$ is locally \emph{monotone},
and combining this with the fact that $f\in W_{\loc}^{1,n}(\Om;\Om')$, by e.g. Mal\'y
\cite[Theorem 3.3, Theorem 4.3]{Mal} we know that
$f$ is differentiable a.e.
Moreover, by \cite[Corollary B]{MaMa} and \cite[Theorem 3.4]{Mal}, such $f$
is absolutely continuous in measure and satisfies the area formula, implying that
\[
\int_W J_f\,d\mathcal L^n=|f(W)|
=\int_W |\det \nabla f|\,d\mathcal L^n
\]
for every open $W\subset \Om$, and so
$|\det \nabla f|=J_f$ a.e. in $\Om$. Thus in \eqref{eq:Df and J 2} we could equivalently
replace $|\det \nabla f|$ with $J_f$.

We will need the following Vitali-Carath\'eodory theorem;
for a proof see e.g. \cite[p. 108]{HKSTbook}.

\begin{theorem}\label{thm:VitaliCar}
	Let $\Om\subset \R^n$ be open and let
	$h\in L^1(\Om)$ be nonnegative. Then there exists a sequence $\{h_i\}_{i=1}^{\infty}$
	of lower semicontinuous functions
	on $\Om$ such that $h\le h_{i+1}\le h_i$ for all $i\in\N$, and
	$h_i\to h$ in $L^1(\Om)$.
\end{theorem}

The theory of $\BV$ mappings that we present next can be found in the monograph
Ambrosio--Fusco--Pallara \cite{AFP}.
As before, let $\Om\subset\R^n$ be an open set. Let $l\in\N$.
A mapping
$f\in L^1(\Omega;\R^l)$ is of bounded variation,
denoted $f\in \BV(\Omega;\R^l)$, if its weak derivative
is an $\R^{l\times n}$-valued Radon measure with finite total variation. This means that
there exists a (unique) Radon measure $Df$
such that for all $\varphi\in C_c^1(\Omega)$, the integration-by-parts formula
\[
\int_{\Omega}f_j\frac{\partial\varphi}{\partial x_k}\,d\mathcal L^n
=-\int_{\Omega}\varphi\,d(Df_j)_k,\quad j=1,\ldots,l,\ k=1,\ldots,n,
\]
holds.
The total variation of $Df$ is denoted by $|Df|$.
The BV norm is defined by
\[
\Vert f\Vert_{\BV(\Om)}:=\Vert f\Vert_{L^1(\Om)}+|Df|(\Om).
\]
If we do not know a priori that a mapping $f\in L^1_{\loc}(\Om;\R^l)$
is a BV mapping, we consider
\begin{equation}\label{eq:definition of total variation}
\Var(f,\Om):=\sup\left\{\sum_{j=1}^{l}\int_{\Om}f_j\dive\varphi_j\,d\mathcal L^n,\,\varphi\in C_c^{1}(\Om;\R^{l\times n}),
\,|\varphi|\le 1\right\}.
\end{equation}
If $\Var(f,\Om)<\infty$, then the $\R^{l\times n}$-valued Radon measure $Df$
exists and $\Var(f,\Om)=|Df|(\Om)$
by the Riesz representation theorem, and $f\in\BV(\Om)$ provided that $f\in L^1(\Om;\R^l)$.
If $E\subset\R^n$ with $\Var(\ch_E,\R^n)<\infty$, we say that $E$ is a set of finite perimeter.

The coarea formula states that for a function $u\in \BV(\Om)$, we have
\begin{equation}\label{eq:coarea}
	|Du|(\Om)=\int_{-\infty}^{\infty}|D\ch_{\{u>t\}}|(\Om)\,dt.
\end{equation}
Here we abbreviate $\{u>t\}:=\{x\in \Om\colon u(x)>t\}$.

The relative isoperimetric inequality states that for every set
of finite perimeter $E\subset \R^n$ and every ball
$B(x,r)$, we have
\begin{equation}\label{eq:relative isoperimetric inequality}
	\min\{\mathcal L^n(B(x,r)\cap E),\mathcal L^n(B(x,r)\setminus E)\}
	\le C_I r|D\ch_E|(B(x,r)),
\end{equation}
where the constant $C_I\ge 1$ only depends on $n$.
The following relative isoperimetric inequality holds on the plane:
for every set of finite perimeter $E\subset \R^2$ and every disk
$B(x,r)$, we have
\begin{equation}\label{eq:rel isop ineq}
	\min\{\mathcal L^2(B(x,r)\cap E),\mathcal L^2(B(x,r)\setminus E)\}
	\le r |D\ch_E|(B(x,r)).
\end{equation}

For $f\in L^1_{\loc}(\Om)$,
we define the precise representative by
\begin{equation}\label{eq:precise representative}
f^*(x):=\limsup_{r\to 0}\vint{B(x,r)}f\,d\mathcal L^n,\quad x\in \Om.
\end{equation}
For $f\in L^1_{\loc}(\Om;\R^n)$, we let
$f^*(x):=(f_1^*(x),\ldots,f_n^*(x))$.

For basic results in the one-dimensional case $n=1$, see \cite[Section 3.2]{AFP}.
If $\Om\subset \R$ is an open interval, we define the pointwise variation of $f\colon \Om\to \R^n$ by
\begin{equation}\label{eq:pointwise variation}
\pV(f,\Om):=\sup \sum_{j=1}^{N-1} |f(x_{j})-f(x_{j+1})|,
\end{equation}
where the supremum is taken over all collections of points $x_1<\ldots<x_N$
in $\Om$. For a general open $\Om\subset \R$, we define $\pV(f,\Om)$ to be $\sum \pV(f,I)$, where the sum
runs over all connected components $I$ of $\Om$.
For every pointwise defined $f\in L^1_{\loc}(\Om;\R^n)$, we have $\Var(f,\Om)\le \pV(f,\Om)$.

Denote by $\pi_n\colon\R^n\to \R^{n-1}$ the orthogonal projection onto $\R^{n-1}$:
for $x=(x_1,\ldots,x_n)\in\R^n$,
\begin{equation}\label{eq:orthogonal}
\pi_n((x_1,\ldots,x_n)):=(x_1,\ldots,x_{n-1}).
\end{equation}
For $z\in\pi_n(\Om)$, we denote the slices of an open set $\Om\subset\R^n$ by
\[
\Om_z\coloneqq \{t\in\R\colon (z,t) \in \Om\}.
\]
We also denote $f_z(t)\coloneqq f(z,t)$ for $z\in\pi_n(\Om)$ and $t\in \Om_z$.
For any continuous $f\in L^1_{\loc}(\Om;\R^n)$, we know that $\Var(f,\Om)$ is at most the sum of
\begin{equation}\label{eq:pointwise variation in a direction}
\int_{\pi_n(\Om)}\pV(f_z,\Om_z)\,d\mathcal L^{n-1}(z)
\end{equation}
and the analogous quantities for the other $n-1$ coordinate directions, see \cite[Theorem 3.103]{AFP}.

The (Sobolev) $1$-capacity of a set $A\subset \R^n$ is defined by
\[
\capa_1(A):=\inf \Vert u\Vert_{W^{1,1}(\R^n)},
\]
where the infimum is taken over Sobolev functions $u\in W^{1,1}(\R^n)$ satisfying
$u\ge 1$ in a neighborhood of $A$.

Given sets $A\subset W\subset \R^n$, where $W$ is open, the relative $p$-capacity is defined by
\[
\rcapa_1(A,W):=\inf \int_{W}|\nabla u|\,d\mathcal L^n,
\]
where the infimum is taken over functions $u\in W_0^{1,1}(W)$ satisfying $u\ge 1$ in a neighborhood
of $A$.
The class $W_0^{1,1}(W)$ is the closure of $C^1_c(W)$ in the $W^{1,p}(\R^n)$-norm.

By \cite[Theorem 3.3]{CDLP}, given a function $u\in\BV(\Om)$, there is a sequence $\{u_j\}_{j=1}^{\infty}$
of functions in $W^{1,1}(\Om)$ such that
\begin{equation}\label{eq:CDLP approx}
	u_j\to u \ \textrm{ in }L^1(\Om),\quad |Du_j|(\Om)\to |Du|(\Om),\
	\textrm{ and }\ u^{\vee}_j(x)\ge u^{\vee}(x)\ \textrm{ for }\mathcal H^{n-1}\textrm{-a.e. }x\in\Om.
\end{equation}

If $B(x,r)$ is a ball with $0<r\le 1$, and $F$ is a measurable set with
$\mathcal L^n(F\cap B(x,r))\le \tfrac 12 \mathcal L^n(B(x,r))$ and $|D\ch_F|(B(x,r))<\infty$,
then by combining e.g. Theorem 5.6  and Theorem 5.15(iii) of \cite{EvGa}, we get
\[
|D\ch_{B(x,r)\cap F}|(\R^n)\le C \Vert \ch_F\Vert_{\BV(B(x,r))}
\]
for some constant $C$ depending only on $n,r$. On the other hand,
by the relative isoperimetric inequality \eqref{eq:relative isoperimetric inequality}, we have
\begin{align*}
	\Vert \ch_F\Vert_{\BV(B(x,r))}
	= \mathcal L^n(F\cap B(x,r))+|D\ch_F|(B(x,r))
	&\le (C_I r+1)|D\ch_F|(B(x,r))\\
	&\le 2C_I |D\ch_F|(B(x,r)),
\end{align*}
since $r\le 1$ and $C_I\ge 1$.
Combining these, we get
\begin{equation}\label{eq:constant C depending only on n}
	|D\ch_{B(x,r)\cap F}|(\R^n)\le C|D\ch_F|(B(x,r)),
\end{equation}
and by a scaling argument we see that in fact $C$ only depends on $n$, not on $r$.

By \cite[Proposition 6.16]{BB}, we know that for a ball $B(x,r)$ and $A\subset B(x,r)$, we have
\begin{equation}\label{eq:capa vs rcapa}
	\frac{\capa_1(A)}{C'(1+r)}
	\le \rcapa_1(A,B(x,2r)),
\end{equation}
where $C'$ is a constant depending only on $n$.

We denote $\omega_{n}:=|B(0,1)|$.

\begin{lemma}\label{lem:capa in small ball}
	Suppose $x\in\R^n$, $0<r<1$, and $A\subset B(x,r)$. Then we have
	\[
	\frac{\mathcal L^n(A)}{\mathcal L^n(B(x,r))}\le \frac{2C_I}{\omega_n}\frac{\capa_1(A)}{r^{n-1}}
	\quad\textrm{and}\quad
	\rcapa_1(A,B(x,2r))\le C\capa_1(A),
	\]
	where $C_I$ is the constant in the relative isoperimetric inequality \eqref{eq:relative isoperimetric inequality},
	and $C$ is a constant depending only on $n$.
\end{lemma}

\begin{proof}
	For both inequalities, we can assume that $\capa_1(A)<\infty$.
	Let $\eps>0$.
	We can choose a function $u\in W^{1,1}(\R^n)$ such that $u\ge 1$ in a neighborhood of $A$,
	and
	\[
	\Vert  u\Vert_{W^{1,1}(\R^n)}<\capa_1(A)+\eps.
	\] 
	By the coarea formula \eqref{eq:coarea}, we then find $0<t<1$ such that
	$\{u>t\}$ contains a neighborhood of $A$, and
	\[
	| D\ch_{\{u>t\}}|(\R^n)\le |Du|(\R^n)\le 	\Vert  u\Vert_{W^{1,1}(\R^n)}<\capa_1(A)+\eps. 
	\]
	Denote $F:=\{u>t\}$.
	
	Case 1: Suppose $\mathcal L^n(F\cap B(x,r))\ge \tfrac 12 \mathcal L^n(B(x,r))$.
	We find $R\ge r$ such that $\mathcal L^n(F\cap B(x,R))= \tfrac 12 \mathcal L^n(B(x,R))$.
	By the relative isoperimetric inequality \eqref{eq:relative isoperimetric inequality}, we have
	\begin{equation}\label{eq:perimeter estimate}
		\begin{split}
			\capa_1(A)+\eps
			> |D\ch_F|(\R^n)
			\ge |D\ch_F|(B(x,R))
			&\ge C_I^{-1} \frac 12 R^{-1} \mathcal L^n(B(x,R))\\
			&=\frac{\omega_n}{2C_I} R^{n-1}\\
			&\ge\frac{\omega_n}{2C_I} r^{n-1}\\
			&\ge \frac{\omega_n}{2C_I} r^{n-1}\frac{\mathcal L^n(F\cap B(x,r))}{\mathcal L^n(B(x,r))}\\
			&\ge \frac{\omega_n}{2C_I} r^{n-1}\frac{\mathcal L^n(A)}{\mathcal L^n(B(x,r))}.
		\end{split}
	\end{equation}
	Letting $\eps\to 0$, we get the first result.
	Defining the cutoff function
	\begin{equation}\label{eq:cutoff}
		\eta(y):=\max\left\{0,1-\frac{1}{r}\dist(y,B(x, r))\right\},\quad y\in\R^n,
	\end{equation}
	for which $\eta=1$ in $B(x,r)$ and $\eta=0$ in $\R^n\setminus B(x,2r)$, we get
	\begin{align*}
		\rcapa_1(A,B(x,2r))
		\le \int_{\R^n} |\nabla \eta|\,d\mathcal L^n
		\le \frac{\omega_n (2r)^n}{r}
		\le 2^{n+1}C_I (\capa_1(A)+\eps)
	\end{align*}
	by the first three lines of \eqref{eq:perimeter estimate}. 
	Letting $\eps\to 0$, we get the second result with $C=2^{n+1}C_I$.
	
	Case 2: Suppose $\mathcal L^n(F\cap B(x,r))< \tfrac 12 \mathcal L^n(B(x,r))$.
	By the relative isoperimetric inequality,
	\begin{align*}
		\capa_1(A)+\eps
		\ge |D\ch_F|(\R^n)
		\ge |D\ch_F|(B(x,r))
		&\ge \frac{1}{C_Ir}\mathcal L^n(F\cap B(x,r))\\
		&\ge \frac{1}{C_Ir}\mathcal L^n(A)\\
		&\ge \frac{\omega_n}{C_I}\mathcal L^n(A)\frac{r^{n-1}}{\mathcal L^n(B(x,r))}.
	\end{align*}
	Letting $\eps\to 0$, we get the first result.
	
	By \eqref{eq:constant C depending only on n}, we get
	\begin{equation}\label{eq:B cap E estimate}
		|D\ch_{B(x,r)\cap F}|(\R^n)\le C|D\ch_F|(B(x,r))\le C\capa_1(A)+C\eps.
	\end{equation}
	By \eqref{eq:CDLP approx}, we find a sequence $\{u_j\}_{j=1}^{\infty}$ in $W^{1,1}(\R^n)$ such that
	$u_j\to \ch_{B(x,r)\cap F}$ in $L^1(\R^n)$, $|Du_j|(\R^n)\to |D\ch_{B(x,r)\cap F}|(\R^n)$,
	and $u_j\ge 1$ a.e.  in a neighborhood of $A$.
	Consider the cutoff function $\eta$ from \eqref{eq:cutoff}.
	We have $u_j\eta \to \ch_{B(x,r)\cap F}$ in $L^1(\R^n)$, $|D(u_j \eta)|(\R^n)\to |D\ch_{B(x,r)\cap F}|(\R^n)$,
	and $u_j\eta \ge 1$ a.e. in a neighborhood of $A$.
	Thus
	\begin{align*}
		\rcapa_1(A,B(x,2r))
		\le \liminf_{j\to\infty}\int_{\R^n}|\nabla (u_j \eta)|\,d\mathcal L^n
		&= |D\ch_{B(x,r)\cap F}|(\R^n)\\
		&\le C\capa_1(A)+C\eps\quad\textrm{by }\eqref{eq:B cap E estimate}.
	\end{align*}
	Letting $\eps\to 0$, we get the second result.
\end{proof}

\begin{definition}\label{def:1 fine topology}
	We say that $A\subset \R^n$ is $1$-thin at the point $x\in \R^n$ if
	\[
	\lim_{r\to 0}\frac{\capa_1(A\cap B(x,r))}{r^{n-1}}=0.
	\]
	We also say that a set $U\subset \R^n$ is $1$-finely open if $\R^n\setminus U$ is $1$-thin at every $x\in U$. Then we define the $1$-fine topology as the collection of $1$-finely open sets on $\R^n$.
	
	We denote the $1$-fine interior of a set $H\subset \R^n$, i.e. the largest $1$-finely open set contained in $H$, by $\fineint H$. We denote the $1$-fine closure of $H$, i.e. the smallest $1$-finely closed set
	containing $H$, by $\overline{H}^1$.
	The $1$-base $b_1 H$ is defined as the set of points
	where $H$ is \emph{not} $1$-thin.
\end{definition}

See \cite[Section 4]{L-FC} for discussion on Definition \ref{def:1 fine topology},
and for a proof of the fact that the
$1$-fine topology is indeed a topology.
In fact, in \cite{L-FC} the criterion
\[
\lim_{r\to 0}\frac{\rcapa_1(A\cap B(x,r),B(x,2r))}{r^{n-1}}=0
\]
for $1$-thinness was used, in the context of more general metric measure spaces.
By \eqref{eq:capa vs rcapa} and Lemma \ref{lem:capa in small ball},
this is equivalent with our current definition in the Euclidean setting.

According to \cite[Corollary 3.5]{L-Fed}, the $1$-fine closure
of $A\subset \R^n$ can be characterized as
\begin{equation}\label{eq:characterization of fine closure}
	\overline{A}^1=A\cup b_1 A.
\end{equation}

\section{Proof of Theorem \ref{thm:K fine basic}}\label{sec:basic}

In this section we prove Theorem \ref{thm:K fine basic}.
We work in $\R^n$ with $n\ge 2$.
First we give the following simple example demonstrating that $K_f$ is generally not a
natural quantity to consider for mappings $f\in W_{\loc}^{1,n}(\R^n;\R^n)$,
let alone mappings of lower regularity.

\begin{example}\label{ex:W11 function}
Let $\{q_j\}_{j=1}^{\infty}$ be an enumeration of points in $\R^n$ with rational coordinates.
Let $f\in W_{\loc}^{1,n}(\R^n;\R^n)$ be such that the first component function is
\[
f_1(x):=\sum_{j=1}^{\infty} 2^{-j}\log \max\{-\log|x-q_j|,1\},\quad x\in \R^n.
\]
Now clearly $\diam f(B(x,r))=\infty$ for every $x\in \R^n$ and $r>0$.
Thus
\[
K_f(x)=\limsup_{r\to 0}\left(\frac{\diam f(B(x,r))^n}{|f(B(x,r))|}\right)^{1/(n-1)}
=\limsup_{r\to 0}\left(\frac{+\infty}{|f(B(x,r))|}\right)^{1/(n-1)}
\]
for every $x\in \R^n$, and so regardless of the value of $|f(B(x,r))|$, the quantity
$K_f$ is either $+\infty$ or undefined.
\end{example}

The Hardy--Littlewood maximal function of a function $u\in L^1_{\loc}(\R^n)$ is defined by
\begin{equation}\label{eq:HL maximal function}
	\mathcal M u(x):=\sup_{r>0}\,\vint{B(x,r)}|u|\,d\mathcal L^n,\quad x\in\R^n.
\end{equation}
We also define a restricted version $\mathcal M_R u(x)$, with $R>0$, by requiring $0<r\le R$ in the supremum.

It is well known, see e.g. \cite[Theorem 1.15]{KLV}, that
\begin{equation}\label{eq:weak estimate}
	|\{x\in \R^n\colon \mathcal Mu(x)>t\}|\le \frac{C_0}{t}\Vert u\Vert_{L^1(\R^n)},\quad t>0,
\end{equation}
for a constant $C_0$ depending only on $n$.

The following weak-type estimate is standard, see e.g. \cite[Theorem 4.18]{EvGa};
in this reference a slightly different definition for capacity is used, but a small modification
of the proof gives the following result.

\begin{lemma}\label{lem:weak standard estimate maximal function}
	Let $u\in \BV(\R^n)$. Then for some constant $C$ depending only on $n$, we have
	\[
	\capa_1(\{\mathcal M u>t\}) \le C\frac{\Vert u\Vert_{\BV(\R^n)}}{t}\quad\textrm{for all }t>0.
	\]
\end{lemma}

We will need the following version of Lemma \ref{lem:weak standard estimate maximal function};
recall also the definition of $\mathcal M_R u$ from above that lemma.

\begin{lemma}\label{lem:weak estimate maximal function}
	Let $u\in L^1(\R^n)$. Then for some constant $C$ depending only on $n$, we have
	\[
	\capa_1(\{\mathcal M_1 u>t\}\cap B(x,1)) \le C\frac{\Vert u\Vert_{\BV(B(x,2))}}{t}\quad\textrm{for all }t>0.
	\]
\end{lemma}
\begin{proof}
	We can assume that $\Vert u\Vert_{\BV(B(x,2))}$ is finite.
	Denote by $Eu$ an extension of $u$ from $B(x,2)$ to $\R^n$
	with $\Vert Eu\Vert_{\BV(\R^n)}\le C' \Vert u\Vert_{\BV(B(x,2))}$,
	for some $C'$ depending only on $n$; see e.g. \cite[Proposition 3.21]{AFP}.
	We estimate
	\begin{align*}
		\capa_1(\{\mathcal M_1 u>t\}\cap B(x,1))
		&= \capa_1(\{\mathcal M_1 Eu>t\}\cap B(x,1))\\
		&\le \capa_1(\{\mathcal M_1 Eu>t\})\\
		&\le C\frac{\Vert Eu \Vert_{\BV(\R^n)}}{t}\quad\textrm{by Lemma }\ref{lem:weak standard estimate maximal function}\\
		&\le CC'\frac{\Vert u \Vert_{\BV(B(x,2))}}{t}.
	\end{align*}
\end{proof}

It is known that Sobolev and BV functions are approximately differentiable
a.e., in the sense of \eqref{eq:approximate differentiability}
below. In the following theorem we show a stronger property,
namely that these functions are also \emph{$1$-finely differentiable} a.e.

Recall the definition of the precise representative from \eqref{eq:precise representative}.

\begin{theorem}\label{thm:fine diff}
	Let $\Om\subset \R^n$ be open and let
	$f\in \BV_{\loc}(\Omega;\R^l)$, with $l\in\N$. Then for a.e. $x\in \Om$ there exists a $1$-finely open
	set $U\ni x$ such that
	\[
	\lim_{U\setminus \{x\}\ni y\to x}\frac{|f^*(y)-f^*(x)-\nabla f(x)(y-x)|}{|y-x|}=0.
	\]
\end{theorem}

\begin{proof}
	Since the issue is local, we can assume that $\Om=\R^n$.
	First assume also that $l=1$.
	At a.e. $x\in\R^n$, we have
	\begin{equation}\label{eq:approximate differentiability}
	\lim_{r\to 0}\,\vint{B(x,r)}\frac{|f(y)-f^*(x)-\langle \nabla f(x), y-x\rangle|}{r}\,d\mathcal L^n(y)=0,
	\end{equation}
	see \cite[Theorem 3.83]{AFP}, as well as
	\[
	\lim_{r\to 0}\,\vint{B(x,r)}|\nabla f(y)-\nabla f(x)|\,d\mathcal L^n(y)=0
	\quad\textrm{and}\quad
	\lim_{r\to 0}\frac{|D^s f|(B(x,r))}{r^n}=0.
	\]
	Consider such $x$.
	Define $L(z):=\langle \nabla u(x),z\rangle$.
	Thus for the scalings
	\begin{equation}\label{eq:scalings}
		f_{x,r}(z):=\frac{f(x+rz)-f^*(x)}{r},\quad z\in B(0,2),
	\end{equation}
	we have
	\[
	f_{x,r}(\cdot)\to L(\cdot)\ \ \textrm{in }L^1(B(0,2))\textrm{ as }r\to 0
	\quad\textrm{and}\quad
	\nabla f_{x,r}(z) = \nabla f(x+rz),\ z\in B(0,2).
	\]
	Then
	\begin{align*}
		|D (f_{x,r}-L)|(B(0,2))
		&=\int_{B(0,2)}|\nabla f_{x,r}(z)-\nabla f(x)|\,d\mathcal L^n(z)+|D^s f_{x,r}|(B(0,2))\\
		&=\frac{1}{r^n}\int_{B(x,2r)}|\nabla f(y)-\nabla f(x)|\,d\mathcal L^n(y)+\frac{|D^s f|(B(x,2r))}{r^n}\\
		&\to 0\quad\textrm{as }r\to 0.
	\end{align*}
	In conclusion, we have the norm convergence
	\begin{equation}\label{eq:convergence of blowups}
		f_{x,r} \to L\quad\textrm{in }\BV(B(0,2)).
	\end{equation}
	Note that $(f^*)_{x,r}=(f_{x,r})^*$ in $B(0,2)$, so we simply use the notation $f^*_{x,r}$.
	Note also that
	\[
	|f^*_{x,r}-L|=|(f_{x,r}-L)^*|\le |f_{x,r}-L|^*\le \mathcal M_1|f_{x,r}-L|,
	\]
	and so for every $j\in\N$ and $t>0$ we get
	\begin{align*}
		&\capa_1(\{z\in B(0,1)\colon |f^*_{x,2^{-j}}(z)-L(z)|>t\})\\
		&\qquad \le \capa_1(\{z\in B(0,1)\colon \mathcal M_1|f_{x,2^{-j}}-L|(z)>t\})\\
		&\qquad \le C\frac{\Vert f_{x,2^{-j}}-L\Vert_{\BV(B(0,2))}}{t}\quad \textrm{by Lemma \ref{lem:weak estimate maximal function}}\\
		&\qquad \to 0\quad\textrm{as }j\to\infty\quad \textrm{by }\eqref{eq:convergence of blowups}.
	\end{align*}
	Thus we can choose numbers $t_j\searrow 0$ such that for the sets
	\[
	D_j:=\{z\in B(0,1)\colon |f^*_{x,2^{-j}}(z)-L(z)|>t_j\},
	\]
	we get $\capa_1(D_j)\to 0$ as $j\to\infty$.
	Define $A_j:=D_j\setminus B(0,1/2)$
	and $A:=\bigcup_{j=1}^{\infty}2^{-j}A_j+x$.
	Now for all $k\in\N$, we have
	\begin{align*}
		\capa_1(A\cap B(x,2^{-k}))
		&\le \sum_{j=k}^{\infty}\capa_1(2^{-j}A_j+x)\\
		&= \sum_{j=k}^{\infty}2^{-j(n-1)}\capa_1(A_j)\\
		&\le  \sum_{j=k}^{\infty}2^{-j(n-1)}\capa_1(D_j)\\
		&\le 2^{-k(n-1)+1}\max_{j\ge k}\capa_1(D_j).
	\end{align*}
	Since $\capa_1(D_j)\to 0$, we obtain
	\[
	\frac{\capa_1(A\cap B(0,2^{-k}))}{2^{-k(n-1)}}\to 0\quad
	\textrm{as }k\to\infty,
	\]
	and so clearly $A$ is $1$-thin at $x$.
	By \eqref{eq:characterization of fine closure}, the $1$-finely open set
	$U:=\R^n\setminus \overline{A}^1$ contains $x$.
	For any $j\in\N$ and $y\in U\cap B(x,2^{-j})\setminus B(x,2^{-j-1})$, we have
	\begin{align*}
		\frac{|f^*(y)-f^*(x)-L(y-x)|}{|y-x|}
		&\le 2\frac{|f^*(y)-f^*(x)-L(y-x)|}{2^{-j}}\\
		&= 2|f^*_{x,2^{-j}}((y-x)/2^{-j})-L((y-x)/2^{-j})|\\
		&\le 2t_j\to 0\quad\textrm{as }j\to\infty,
	\end{align*}
	and so
	\[
	\lim_{U\ni y\to x}\frac{|f^*(y)-f^*(x)-\langle\nabla f(x), y-x\rangle|}{|y-x|}=0.
	\]
	
	Finally, the generalization to the case $l\in\N$ is immediate, since the intersection of
	a finite number of $1$-finely open sets is still $1$-finely open.
\end{proof}

Given $f\in W^{1,1}_{\loc}(\Om;\R^n)$,
note that the weak gradient $\nabla f$ is a function in $L^1_{\loc}(\Om;\R^{n\times n})$
and thus may be understood
to be an equivalence class rather than a pointwise defined function.
Below, we sometimes consider $\nabla f$ at a given point; for this
we can understand $\nabla f$ to be well defined everywhere by using the above theorem and by
defining $\nabla f$ to be zero in the exceptional set.

We restate the following definition already given in the Introduction.

\begin{definition}
	Let $f\colon \R^n\to [-\infty,\infty]^{n}$ and $U\subset \R^n$. Then we let
	\[
	K_{f,U}(x,r):=\left(\frac{\diam f(B(x,r)\cap U)^n}{|f(B(x,r))|}\right)^{1/(n-1)}\quad\textrm{and}\quad
	K_{f}^{\mathrm{fine}}(x):=\inf\limsup_{r\to 0}K_{f,U}(x,r),
	\]
	where the infimum is taken over $1$-finely open sets $U\ni x$.
	If $|f(B(x,r))|=0$, then we interpret $K_{f,U}(x,r)=\infty$.
\end{definition}

\begin{proof}[Proof of Theorem \ref{thm:K fine basic}]
	Let $f\in W^{1,1}_{\loc}(\R^n;\R^n)$; since the claim is local, we can assume that in fact
	$f\in W^{1,1}(\R^n;\R^n)$.
	Using e.g. \cite[Theorem 6.13]{EvGa},
	we find a Lipschitz mapping $\widehat{f}\in \Lip(\R^n;\R^n)$ such that the complement of the set
	\[
	H:=\{z\in \R^n\colon (\widehat{f})^*(z)= f^*(z)
	\  \textrm{and}\  \nabla \widehat{f}(z)= \nabla f(z)\}
	\]
	has arbitrarily small Lebesgue measure.
	By e.g. \cite[Lemma 2.74]{AFP}, $\mathcal L^n$-almost all of the set $\{z\in\R^n\colon \det \nabla \widehat{f}(z)\neq 0\}$ can be covered
	by compact sets $\{K_j\}_{j=1}^{\infty}$ such that $\widehat{f}$ is injective
	in each $K_j$.
	Consider a point $x\in\R^n$ for which $\det \nabla f(x)\neq 0$.
	Since the theorem is formulated as an ``a.e.'' result,
	we can assume that
	\[
	\lim_{r\to 0}\frac{|B(x,r)\cap H|}{|B(x,r)|}=1
	\quad\textrm{and}\quad
	\lim_{r\to 0}\frac{|B(x,r)\cap K_j|}{|B(x,r)|}=1
	\quad\textrm{for some }j,
	\]
	that $f$ is $1$-finely differentiable as in Theorem \ref{thm:fine diff}
	so that we find a $1$-finely open set $U\ni x$ such that
	\begin{equation}\label{eq:1 fine diff}
	\lim_{U\setminus \{x\}\ni y\to x}\frac{|f^*(y)-f^*(x)-\nabla f(x)(y-x)|}{|y-x|}=0,
	\end{equation}
	and that $x$ is a Lebesgue point of $\nabla f$:
	\begin{equation}\label{eq:Lebesgue p}
	\lim_{r\to 0}\,\vint{B(x,r)}| \nabla f(y)-\nabla f(x)| \,d\mathcal L^n(y)=0.
	\end{equation}
	For the scalings
	\[
	f_{r}(z):=\frac{f^*(x+rz)-f^*(x)}{r},\quad z\in B(0,1),
	\]
	we have $\nabla f_{r}(z) = \nabla f(x+rz)$, with $z\in B(0,1)$, and thus
	by \eqref{eq:Lebesgue p},
	\begin{equation}\label{eq:B01 Lebesgue}
	\lim_{r\to 0}\int_{B(0,1)}|\nabla f_{r} -\nabla f(x)|\,d\mathcal L^n =0.
	\end{equation}
	Fix $\eps>0$.
	Let
	\[
	D^r:=\{z\in B(0,1)\colon | \det(\nabla f_r(z) -\nabla f(x))| < \eps |\det \nabla f(x)|\}.
	\]
	By \eqref{eq:B01 Lebesgue}, we also have
	$|B(0,1)\setminus D^r|<\omega_n\eps$ for sufficiently small $r>0$.
	Let
	\[
	H_r:=r^{-1}(H-x).
	\]
	For sufficiently small $r>0$ we have in total
	\begin{equation}\label{eq:small Lebesgue measure estimate}
	|B(0,1)\setminus D^r|+|B(0,1)\setminus H_r|+|B(0,1)\setminus (K_j)_r|<\omega_n \eps.
	\end{equation}
	In the set $D^r\cap H_r\cap (K_j)_r$, we have
	\begin{equation}\label{eq:det nabla estimate}
	\begin{split}
	|\det \nabla \widehat{f}_r|
	=|\det \nabla f_r|
	\ge |\det \nabla f(x)| -|\det (\nabla f_r-\nabla f(x))|
	\ge (1-\eps)|\det \nabla f(x)|.
	\end{split}
	\end{equation}
	Now by the area formula, see e.g. \cite[Theorem 2.71]{AFP}, we get
	\begin{align*}
	|f_r(B(0,1))|
	&\ge |f_r(D^r\cap H_r\cap (K_j)_r)|\\
	&=|\widehat{f}_r(D^r\cap H_r\cap (K_j)_r)|\\
	&=\int_{D^r\cap H_r\cap (K_j)_r}|\det \nabla \widehat{f}_r|\,d\mathcal L^n\\
	&\ge (1-\eps)\int_{D^r\cap H_r\cap (K_j)_r}|\det \nabla f(x)|\,d\mathcal L^n
	\quad \textrm{by }\eqref{eq:det nabla estimate}\\
	&\ge (1-\eps)^2 \omega_n |\det \nabla f(x)|\quad \textrm{by }\eqref{eq:small Lebesgue measure estimate}.
	\end{align*}
	Thus
	\[
	|f^*(B(x,r))|
	= r^n|f_r(B(0,1))|
	\ge (1-\eps)^2 \omega_n r^n |\det \nabla f(x)|.
	\]
	Thus using also the fine differentiability \eqref{eq:1 fine diff}, we get
	\begin{align*}
	\limsup_{r\to 0}\frac{\diam f^*(B(x,r)\cap U)^n}{|f^*(B(x,r))|}
	&\le \limsup_{r\to 0}\frac{2^n\Vert \nabla f(x)\Vert^n 
		 r^n}{(1-\eps)^2 \omega_n r^n |\det \nabla f(x)|}\\
	&=\frac{2^n\Vert \nabla f(x)\Vert^n}{(1-\eps)^2 \omega_n |\det \nabla f(x)|}.
	\end{align*}
	It follows that
	\[
	(1-\eps)^2 K_{f^*}^{\mathrm{fine}}(x) ^{n-1}|\det \nabla f(x)|
	\le \frac{2^n}{\omega_n}\Vert \nabla f(x)\Vert^n.
	\]
	Letting $\eps\to 0$, we get the result.
\end{proof}

\section{Proof of Theorem \ref{thm:main}}

In this section we prove our main Theorem \ref{thm:main}.
At first we work in $\R^n$ with $n\ge 2$, but in our main results we need $n=2$.
We start with the following simple lemma.

\begin{lemma}\label{lem:tightness}
Assume $\Om\subset \R^n$ is open, $f\in W^{1,1}_{\loc}(\Om;\R^n)$ is continuous, $x\in\Om$,
and suppose $U\ni x$ is a $1$-finely open set such that
\begin{equation}\label{eq:1 fine diff ass}
\lim_{U\setminus \{x\}\ni y\to x}\frac{|f(y)-f(x)-\nabla f(x) (y-x)|}{|y-x|}=0.
\end{equation}
Then
\[
\limsup_{r\to 0}\frac{\diam f(B(x,r)\cap U)^n}{|B(x,r)|}
= \frac{2^n}{\omega_n} \Vert \nabla f(x)\Vert^n.
\]
\end{lemma}
\begin{proof}
By Lemma \ref{lem:capa in small ball}, we have
\[
\lim_{r\to 0}\frac{|B(x,r)\setminus U|}{|B(x,r)|}=0,
\]
and so for the linear mapping $L(y):=\nabla f(x)(y-x)$, we clearly have
\[
	\lim_{r\to 0}\frac{\diam L(B(x,r)\cap U)}{r}=2\Vert \nabla f(x)\Vert.
\]
Then by the fine differentiability \eqref{eq:1 fine diff ass},  we also have
\begin{align*}
	\lim_{r\to 0}\frac{\diam f(B(x,r)\cap U)}{r}= 2\Vert\nabla f(x)\Vert,
\end{align*}
and so the claim follows.
\end{proof}

Now we show that the following version of Theorem \ref{thm:K fine basic} holds
when $f$ is additionally assumed to be a homeomorphism;
recall \eqref{eq:Jacobian}.

\begin{proposition}\label{prop:homeo}
	Let $\Om,\Om'\subset \R^n$ be open and let
	$f\in W_{\loc}^{1,1}(\Om;\Om')$ be a homeomorphism. Then we have
	\[
		K_f^{\mathrm{fine}}(x)^{n-1}J_f(x)= \frac{2^n}{\omega_n} \Vert \nabla f(x)\Vert^n
		\quad \textrm{for a.e. }x\in \Om\textrm{ for which }K_f^{\mathrm{fine}}(x)<\infty,
	\]
	and $K_f(x)=K_f^{\mathrm{fine}}(x)$ for a.e. $x\in \Om$ where $f$ is differentiable
	and $0<J_f(x)<\infty$.
\end{proposition}
\begin{proof}
	Consider $x\in\Om$ for which $K_f^{\mathrm{fine}}(x)<\infty$.
	Thus we find a $1$-finely open set $V\ni x$ such that
	\begin{align*}
		\limsup_{r\to 0}\frac{\diam f(B(x,r)\cap V)^n}{|f(B(x,r))|}<\infty.
	\end{align*}
	Excluding a $\mathcal L^n$-negligible set we can also assume that
	$J_f(x)<\infty$ exists as a limit
	(recall the discussion after \eqref{eq:Jacobian}),
	and that
	we find a $1$-finely open set $U\ni x$ with
	\[
		\lim_{U\setminus\{x\}\ni y\to x}\frac{|f(y)-f(x)-\nabla f(x)(y-x)|}{|y-x|}=0,
	\]
	recall Theorem \ref{thm:fine diff}.
	To prove one inequality, we estimate
	\begin{align*}
		K_f^{\mathrm{fine}}(x)^{n-1}J_f(x)
		&\le \limsup_{r\to 0}\frac{\diam f(B(x,r)\cap V\cap U)^n}{|f(B(x,r))|}
		\lim_{r\to 0}\frac{|f(B(x,r))|}{|B(x,r)|}\\
		&= \limsup_{r\to 0}\frac{\diam f(B(x,r)\cap V\cap U)^n}{|f(B(x,r))|}
		\frac{|f(B(x,r))|}{|B(x,r)|}\\
		&\le \limsup_{r\to 0}\frac{\diam f(B(x,r)\cap U)^n}{|B(x,r)|}\\
		&= \frac{2^n}{\omega_n} \Vert \nabla f(x)\Vert^n
	\end{align*}
	by Lemma \ref{lem:tightness}.
	
	Then we prove the opposite inequality.
	Let $\eps>0$.
	We can choose the $1$-finely open set $V\ni x$ such that
	\begin{align*}
	K_f^{\mathrm{fine}}(x)^{n-1}
	&> \limsup_{r\to 0}\frac{\diam f(B(x,r)\cap V)^n}{|f(B(x,r))|}-\eps\\
	&\ge \limsup_{r\to 0}\frac{\diam f(B(x,r)\cap V\cap U)^n}{|f(B(x,r))|}-\eps.
	\end{align*}
	Then
	\begin{align*}
	K_f^{\mathrm{fine}}(x)^{n-1}J_f(x)
	&\ge \left(\limsup_{r\to 0}\frac{\diam f(B(x,r)\cap V\cap U)^n}{|f(B(x,r))|}-\eps\right)
	\lim_{r\to 0}\frac{|f(B(x,r))|}{|B(x,r)|}\\
	&= \limsup_{r\to 0}\frac{\diam f(B(x,r)\cap V\cap U)^n}{|B(x,r)|}-\lim_{r\to 0}\eps\frac{|f(B(x,r))|}{|B(x,r)|}\\
	&= \frac{2^n}{\omega_n} \Vert \nabla f(x)\Vert^n-\eps J_f(x)
	\end{align*}
by Lemma \ref{lem:tightness}. Letting $\eps\to 0$, we get the other inequality.

If $f$ is differentiable at $x\in\Om$ and $0<J_f(x)<\infty$,
we can again assume that
$J_f(x)$ exists as a limit, and
then we also have
\begin{align*}
	K_f(x)^{n-1}J_f(x)
	&= \limsup_{r\to 0}\frac{\diam f(B(x,r))^n}{|f(B(x,r))|}
	\lim_{r\to 0}\frac{|f(B(x,r))|}{|B(x,r)|}\\
	&= \limsup_{r\to 0}\frac{\diam f(B(x,r))^n}{|B(x,r)|}\\
	&= \frac{2^n}{\omega_n} \Vert \nabla f(x)\Vert^n,
\end{align*}
and so
\[
K_f(x)^{n-1}
=\frac{2^n}{\omega_n} \frac{\Vert \nabla f(x)\Vert^n}{J_f(x)}
=K_f^{\mathrm{fine}}(x)^{n-1},
\]
where we also used the first part of the proposition, which is applicable since
$K_f^{\mathrm{fine}}(x)\le K_f(x)<\infty$.
\end{proof}

We note that equation \eqref{eq:Kf n minus one} in the Introduction can be proved similarly to
Proposition \ref{prop:homeo}.

We will use Whitney-type coverings consisting of disks.
As with balls so far, a disk is always understood to be open unless otherwise specified.

\begin{lemma}\label{lem:Whitney}
	Let $A\subset   D\subset W$, where $W\subset \R^2$ is an open set
	and $A$ is dense in $D$.
	Given a scale $0<R<\infty$, there exists a finite or countable Whitney-type covering $\{B_k=B(x_k,r_k)\}_k$ of $A$ in $W$,
	with $x_k\in A$, $r_k\le R$, and the following properties:
	\begin{enumerate}
		\item $B_k\subset W$ and $D\subset \bigcup_{k}\tfrac 12 B_k$,
		\item If $B_k\cap B_l\neq \emptyset$, then $r_{k}\le 2r_{l}$;
		\item The disks $B_k$ can be divided into
		$6400$ collections of pairwise disjoint disks.
	\end{enumerate}
\end{lemma}

\begin{proof}
	For every $x\in A$, let $r_x:=\min\{R,\tfrac{1}{4} \dist(x,\R^n\setminus W)\}$.
	Consider the covering $\{B(x,\frac{1}{10}r_x)\}_{x\in A}$.
	Clearly this is also a covering of $D$.
	By the $5$-covering theorem (see e.g. \cite[Theorem 1.24]{EvGa}),
	we can pick an at most countable
	collection of pairwise disjoint disks $B(x_k,\tfrac{1}{10}r_k)$ such that the disks
	$B(x_k,\tfrac{1}{2} r_k)$ cover $D$.
	Denote $B_k=B(x_k,r_k)$. We have established property (1).
	
	Suppose $B_k\cap B_l\neq \emptyset$. If $r_l=\tfrac{1}{4} \dist(x_l,\R^n\setminus W)$, then
	\begin{align*}
		4r_l
		=\dist(x_l,\R^n\setminus W)
		\ge \dist(x_k,\R^n\setminus W)-r_l-r_k
		\ge 4r_k-r_l-r_k
		= 3r_k-r_l,
	\end{align*}
	and so we get $2r_l\ge r_k$.
	If $r_l=R$, then $r_k\le R=r_l$. Thus we get property (2).
	
	For a given $k$, denote by $I$ the set of those indices $l\in I$ such that $B_l\cap B_k\neq \emptyset$.
	For all $l\in I$, by (2) we have $r_k\le 2r_l$ and $\tfrac{1}{10}B_l\subset 4B_k$, and so
	\[
	\sum_{l\in I}400^{-1}\pi r_k^2
	\le \sum_{l\in I}100^{-1}\pi r_l^2
	= \sum_{l\in I} \mathcal L^2(\tfrac{1}{10} B_l)
	\le \mathcal L^2(4B_k)
	= 16 \pi r_k^2,
	\]
	and so the cardinality of $I$ is at most $6400$, and we obtain (3).
\end{proof}

\begin{lemma}\label{lem:H and capa}
	Let $A\subset \R^2$. Then $\mathcal H^1_{\infty}(A)\le 10\capa_1(A)$.
\end{lemma}
\begin{proof}
	We can assume that $\capa_1(A)<\infty$.
	Let $\eps>0$.
	We find a function $u\in W^{1,1}(\R^2)$ such that
	$u\ge 1$ in a neighborhood of $A$, and
	\[
	\int_{\R^2} |\nabla u|\,d\mathcal L^2\le \capa_1(A)+\eps.
	\]
	Here $u\in W^{1,1}(\R^2)\subset  \BV(\R^2)$ with $|Du|(\R^2)= \int_{\R^2} |\nabla u|\,d\mathcal L^2$,
	and then by the coarea formula \eqref{eq:coarea} we find
	a set $E:=\{u>t\}$ for some $0<t<1$, for which
	\[
	|D\ch_{E}|(\R^2)\le |Du|(\R^2) \le \capa_1(A)+\eps,
	\]
	and  $A$ is contained in the interior of $E$.
	Then necessarily $|E|<\infty$, and for every $x\in A$ we find $r_x>0$ such that
	\[
	\frac{|B(x,r_x)\cap E|}{|B(x,r_x)|}=\frac{1}{2}.
	\]
	From the relative isoperimetric inequality \eqref{eq:rel isop ineq}, we get
	\[
	\frac{\pi r_x^2}{2}
	=\min\{|B(x,r_x)\cap E|,|B(x,r_x)\setminus E|\}
	\le r_x|D\ch_{E}|(B(x,r_x)).
	\]
	In particular, the radii $r_x$ are uniformly bounded from above by $(2/\pi)|D\ch_E|(\R^2)$.
	By the $5$-covering theorem (see e.g. \cite[p. 60]{HKSTbook}),
	we can choose a finite or countable collection $\{B(x_j,r_j)\}_{j}$ of pairwise disjoint
	disks such that
	the disks $B(x_j,5r_j)$ cover $A$.
	Then
	\begin{align*}
		\mathcal H_{\infty}^1(A)
		\le \sum_{j}10r_j
		\le \frac{20}{\pi}\sum_{j}|D\ch_{E}|(B(x_j,r_j))
		\le 10 |D\ch_{E}|(\R^2)
		\le 10(\capa_1(A)+\eps).
	\end{align*}
	Letting $\eps\to 0$, we get the result.
\end{proof}

For $x\in\R^2$, let $p(x):=|x|$.

\begin{lemma}\label{lem:polar projection}
	Let $A\subset\R^2$. Then we have $\mathcal L^1(p(A))\le 10\capa_1(A)$.
\end{lemma}
\begin{proof}
	Note that $p$ is a $1$-Lipschitz function. Thus we estimate
	\[
	\mathcal L^1(p(A))
	=\mathcal H^1_{\infty}(p(A))
	\le \mathcal H^1_{\infty}(A)
	\le 10\capa_1(A)
	\]
	by Lemma \ref{lem:H and capa}.
\end{proof}

The following theorem is a more general version of our main Theorem \ref{thm:main}.
Note in particular that the function $K_f^{\mathrm{fine}}$
is not generally known to be measurable; in Theorem \ref{thm:main} measurability
is an assumption implicitly contained in the fact that $K_f^{\mathrm{fine}}\in L^{p^*/2}_{\loc}(\R^n)$.

\begin{theorem}\label{thm:main text}
	Let $\Om,\Om'\subset \R^2$ be open sets with $|\Om'|<\infty$,
	and let $f\colon \Om\to \Om'$ be a homeomorphism.
	Let $1\le p\le 2$.
	Suppose $K_f^{\mathrm{fine}}<\infty$ outside a set $E\subset \Om$
	such that for a.e. line $L$ parallel to a coordinate axis,
	$E\cap L$ is at most countable. 
	Suppose also that there is $h\ge K_f^{\mathrm{fine}}$ such that 
	$h\in L^{p^*/2}(\Om)$.
	Then $f\in D^{p}(\Om;\R^2)$,
	and in the case $p=2$  we obtain that $f$ is quasiconformal with
	$K_f(x)=K_f^{\mathrm{fine}}(x)$ for a.e. $x\in\Om$ and
	\begin{equation}\label{eq:Ln energy thm}
		\Vert \nabla f(x)\Vert ^2
		\le \frac{\pi}{4}\Vert K_f^{\mathrm{fine}}\Vert_{L^{\infty}(\Om)}J_f(x)
		\quad \textrm{for a.e. }x\in \Om.
	\end{equation}
\end{theorem}

Recall that here $D^p(\Om;\R^2)$ is the Dirichlet space, that is, $f$ is not necessarily
in $L^p(\Om;\R^2)$, only in $L^1_{\loc}(\Om;\R^2)$.

\begin{proof}
	We can assume that $\Om$ is nonempty, and at first we also assume that
	$\Om$ is bounded.
	The crux of the proof is to show $D^1$-regularity. For this, we use the fact that
	$h\in L^{p^*/2}(\Om)\subset L^1(\Om)$.
	First we assume also that $h$ is lower semicontinuous.
	
		Fix $0<\eps\le 1$.
		To every $x\in\Om\setminus E$, there corresponds a $1$-finely open set $U_x\ni x$
		for which
		\[
		\lim_{r\to 0}\frac{\capa_1(B(x,r)\setminus U_{x})}{r}=0
		\]
		and
		\[
		\limsup_{r\to 0}\frac{\diam f(B(x,r)\cap U_x)}{|f(B(x,r))|^{1/2}}
		< K_f^{\mathrm{fine}}(x)^{1/2}+\eps.
		\]
		For each $j\in \N$, let $A_j$ consist of points $x\in\Om\setminus E$ for which
		\begin{equation}\label{eq:1 over j capa}
		\sup_{0<r\le 1/j}\frac{\capa_1(B(x,r)\setminus U_{x})}{r}
		<\frac{1}{20}
		\end{equation}
		and
		\begin{equation}\label{eq:Kf estimate}
			\sup_{0<r\le 1/j}\frac{\diam f(B(x,r)\cap U_x)}{|f(B(x,r))|^{1/2}}
			<K_f^{\mathrm{fine}}(x)^{1/2}+\eps,
		\end{equation}
		and also (recall the lower semicontinuity of $h$)
		\begin{equation}\label{eq:h estimate with a}
			K_f^{\mathrm{fine}}(x)< h(y)+\eps\quad \textrm{for all }y\in B(x,1/j)\cap \Om.
		\end{equation}
		We have $\Om=\bigcup_{j=1}^{\infty}A_j\cup E$.
		For $j=1,2\ldots$, we inductively define
		\[
		D_j:=\Om\cap\overline{A_j\setminus \cup_{l=1}^{j-1}D_l}  \setminus \bigcup_{l=1}^{j-1}D_l;
		\]
		note that we do not know the sets $A_j$ to be measurable but the sets $D_j$ are Borel sets.
		Moreover, the sets $D_j$ are disjoint and
		$D_j \supset A_j\setminus \bigcup_{l=1}^{j-1}D_l$, so that
		$\Om= E\cup \bigcup_{j=1}^{\infty}D_j$.
		We can pick open sets $W_j\supset D_j$ such that $W_j\subset \Om$ and
		\begin{equation}\label{eq:Wj Dj}
		\int_{W_j}(h+2\eps)\,d\mathcal L^2
		\le \int_{D_j}(h+2\eps)\,d\mathcal L^2+2^{-j}\eps
		\end{equation}
		and such that
		\begin{equation}\label{eq:Wj Dj image side}
		|f(W_j)|\le |f(D_j)|+2^{-j}\eps.
		\end{equation}

		Fix $R>0$.
		For each $k\in\N$, using Lemma \ref{lem:Whitney} we
		take a Whitney-type covering
		\[
		\{B_{j,k}=B(x_{j,k},r_{j,k})\}_k
		\]
		of $A_{j}\setminus \bigcup_{l=1}^{j-1}D_l$ in $W_{j}$ at scale
		$\min\{R,1/j\}$.
		By Lemma \ref{lem:Whitney}(1), we know that $D_j\subset \bigcup_{k}\tfrac 12 B_{j,k}$ and so
		\begin{equation}\label{eq:Whitney balls contain also D}
			\Om\setminus E\subset \bigcup_{j,k}\tfrac 12 B_{j,k}.
		\end{equation}
		For each point $x_{j,k}$, there is the corresponding $1$-finely open set $U_{x_{j,k}}$.
		Denote $U_{j,k}:= U_{x_{j,k}}\cap B_{j,k}$.
		For any $x\in\R^2$ and $r>0$, denote a circle by $S(x,r)$.
			From \eqref{eq:1 over j capa} and Lemma \ref{lem:polar projection}, we obtain that
		there exists $s_{j,k}\in (r_{j,k}/2,r_{j,k})$ such that $S(x_{j,k},s_{j,k})\subset U_{j,k}$.
		
		Define
		\[
		g:=2\sum_{j,k} 
		\frac{\diam f(U_{j,k}) }{r_{j,k}}\ch_{B_{j,k}}.
		\]
		
		By assumption, for almost every line $L$ in the direction of a coordinate axis,
		$L\cap E$ is at most countable.
		Take a line segment $\gamma\colon [0,\ell]\to L\cap \Om$ in such a line $L$,
		with length  $\ell>0$.
		We denote also the image of $\gamma$ by the same symbol.
		If $\ell\ge R$, we have
		\begin{equation}\label{eq:ug inequality}
		\int_\gamma g\,ds	
		\ge 2\sum_{j,k,\,\gamma\cap \tfrac 12 B_{j,k}\neq \emptyset}\int_\gamma \frac{\diam f(U_{j,k}) }{r_{j,k}}\ch_{B_{j,k}}\,ds
		\ge \sum_{j,k,\,\gamma\cap \tfrac 12 B_{j,k}\neq \emptyset}\diam f(U_{j,k}).
		\end{equation}
		By \eqref{eq:Whitney balls contain also D}, we have
		\[
		\gamma\setminus E \subset \bigcup_{j,k}\tfrac 12 B_{j,k}.
		\]
		Let $0<\delta<R$. Since  $L\cap E$ is at most countable, using the continuity
		of $f$ we find a finite or countable collection of disks $\{B_l\}$ 
		intersecting $\gamma\cap E$ such that
		$\overline{B_l}\subset \Om$ and
		\[
		\gamma\cap E\subset \bigcup_l B_l
		\quad\textrm{and}\quad
		\sum_l \diam f(\overline{B_l})<\delta.
		\]
		Since the disks
		$\tfrac 12 B_{j,k}$ and $B_l$ are open, there is in fact a finite number of
		them covering $\gamma$. 
		Thus there are finite index sets $I_1$ and $I_2$ such that
		every disk $\tfrac 12 B_{j,k}$ with $(j,k)\in I_1$ intersects $\gamma$ and
		\[
		\gamma
		\subset \bigcup_{(j,k)\in I_1}\tfrac 12 B_{j,k}\cup \bigcup_{l\in I_2} B_l
		\subset \bigcup_{(j,k)\in I_1} B(x_{j,k},s_{j,k})\cup \bigcup_{l\in I_2} B_l.
		\]
		We find subsets $J_1\subset I_1$ and $J_2\subset I_2$ such that
		among the disks $B(x_{j,k},s_{j,k})$, $(j,k)\in J_1$, $B_l$, $l\in J_2$,
		no disk is fully contained in another disk, and we still have
		\[
		\gamma
		\subset \bigcup_{(j,k)\in J_1}B(x_{j,k},s_{j,k})
		\cup \bigcup_{l\in J_2} B_l.
		\]
		Relabel the disks $B(x_{j,k},s_{j,k})$, $(j,k)\in J_1$, and $B_l$, $l\in J_2$,
		as $B(y_1,r_1),\ldots,B(y_M,r_M)$.
		We can assume that $\gamma$ is in the $x_2$-coordinate direction.
		Denote by $z_m$ the $x_2$-coordinate of the point in $\gamma\cap \overline{B}(y_m,r_m)$
		with the smallest $x_2$-coordinate.
		We can assume that the disks $B(y_1,r_1),\ldots,B(y_M,r_M)$
		are ordered such that $z_1\le \ldots \le z_M$.
		Since these disks cover $\gamma$,
		for each $m=1,\ldots,M-1$ we have that $\overline{B}(y_{m+1},r_{m+1})$ necessarily intersects
		$\bigcup_{m'=1}^m\overline{B}(y_{m'},r_{m'})\cap \gamma$, and since none of the disks
		$\overline{B}(y_{m'},r_{m'})$, $m'\in\{1,\ldots,M\}$ is contained in another one,
		in fact $S(y_{m+1},r_{m+1})$ necessarily intersects $\bigcup_{m'=1}^m S(y_{m'},r_{m'})$.
		In total, $\bigcup_{m=1}^M S(y_{m},r_{m})$ is a connected set.
		It follows that
		\begin{equation}\label{eq:big connected set}
		\diam\left(\left(\bigcup_{m=1}^{M}f(S(y_m,r_m))\right)\right)
		\le \sum_{m=1}^{M}\diam f(S(y_m,r_m)).
		\end{equation}
		Denote by $\omega$ a modulus of continuity of $f$ at the end points of the curve $\gamma$.
		Then
		\begin{align*}
			|f(\gamma(0))-f(\gamma(\ell))|
			&\le \diam\left(\left(\bigcup_{m=1}^{M}f(S(y_m,r_m))\right)\right)+2\omega(R)\\
			&\le \sum_{m=1}^M\diam f(S(y_m,r_m))+2\omega(R)
			\quad\textrm{by }\eqref{eq:big connected set}\\
			&\le \sum_{j,k,\,\gamma\cap \tfrac 12 B_{j,k}\neq \emptyset}
			\diam f(S(x_{j,k},s_{j,k}))+\sum_l \diam f(\overline{B_l})+2\omega(R)\\
			&\le \sum_{j,k,\,\gamma\cap \tfrac 12 B_{j,k}\neq \emptyset}\diam f(U_{j,k})+\sum_l \diam f(\overline{B_l})+2\omega(R)\\
			&\le \int_\gamma g\,ds +\delta+2\omega(R)
			\quad\textrm{by }\eqref{eq:ug inequality}.
			\end{align*}
	Letting $\delta\to 0$, we get
	\begin{equation}\label{eq:ug estimate}
			|f(\gamma(0))-f(\gamma(\ell))|
			\le \int_\gamma g\,ds +2\omega(R).
	\end{equation}
		By Young's inequality we have for any $b_1,b_2\ge 0$ and $0<\kappa\le 1$ that
		\begin{equation}\label{eq:generalized Young}
			\begin{split}
			b_1b_2=\kappa^{1/2} b_1 \kappa^{-1/2} b_2
			&\le \frac{1}{2}\kappa b_1^2 + \frac{1}{2}\kappa^{-1} b_2^{2}.
			\end{split}
		\end{equation}	
For every $j\in\N$, we estimate
\begin{align*}
	2\pi^{-1}\sum_{k}\frac{\diam f(U_{j,k})}{r_{j,k}} | B_{j,k}|
	&  \le 2\sum_{k}\diam f(U_{j,k})r_{j,k}\\
	 &   \le 2\sum_{k} |f(B_{j,k})|^{1/2}
	 \left(K_f^{\mathrm{fine}}(x_{j,k})^{1/2}+\eps\right)r_{j,k}
	 \quad\textrm{by }\eqref{eq:Kf estimate}\\
	 &  \le\kappa\sum_{k} |f(B_{j,k})|
	+2\kappa^{-1}\sum_{k} (K_f^{\mathrm{fine}}(x_{j,k})+\eps)r_{j,k}^{2}
	\quad\textrm{by }\eqref{eq:generalized Young}.
\end{align*}
Using \eqref{eq:h estimate with a}, we estimate further
\begin{align*}
	&2\pi^{-1} \sum_{k}\frac{\diam f(U_{j,k})}{r_{j,k}} | B_{j,k}|\\
	&\   \le \kappa\sum_{k} |f(B_{j,k})|+
	2 \kappa^{-1}\sum_{k} \int_{B_{j,k}}(h+2\eps)\,d\mathcal L^2\\
	&\   \le 6400 \kappa |f(W_j)|\quad\textrm{by Lemma }\ref{lem:Whitney}(3)\textrm{ and injectivity}\\
	&\qquad  +
	12800 \kappa^{-1} \int_{W_{j}}(h+2\eps)\,d\mathcal L^2
	\quad \textrm{by Lemma }\ref{lem:Whitney}(3)\\
	&\   \le 6400\kappa (|f(D_j)|+2^{-j}\eps)+\\
	&\qquad +12800 \kappa^{-1} 
	\left(\int_{D_j}(h+2\eps)\,d\mathcal L^2+2^{-j}\eps\right)
	\quad\textrm{by }\eqref{eq:Wj Dj},\eqref{eq:Wj Dj image side}.
\end{align*}
It follows that for every $j\in\N$,
\begin{equation}\label{eq:gj estimate with sum}
	\begin{split}
		\pi^{-1}\int_{\Om}g\,d\mathcal L^2
		&= 2\pi^{-1}\sum_{j,k}\frac{\diam f(U_{j,k})}{r_{j,k}} \mathcal |B_{j,k}|\\
			&\le 6400\kappa \sum_{j=1}^{\infty}(|f(D_j)|+2^{-j}\eps)+\\
		&\qquad +12800 \kappa^{-1} 
		\sum_{j=1}^{\infty}\left(\int_{D_j}(h+2\eps)\,d\mathcal L^2+2^{-j}\eps\right)\\
		& \le 6400\kappa (|f(\Om)|+\eps)
		+ 12800\kappa^{-1} \left(\int_{\Om}(h+2\eps)\,d\mathcal L^2+\eps\right)\\
		&  \le 6400\kappa |f(\Om)| +12800\kappa^{-1}\int_{\Om} (h+2\eps) \,d\mathcal L^2
		+19200\kappa^{-1}\eps.
	\end{split}
\end{equation}

We can pick functions $g$ as above, with the choices $R=1/i$, to obtain a sequence
$\{g_i\}_{i=1}^{\infty}$.
Recall the definition of pointwise variation from \eqref{eq:pointwise variation},
as well as \eqref{eq:orthogonal}.
By \eqref{eq:ug estimate}, for $\mathcal L^{1}$-a.e. $z\in\pi_2(\Om)$  we get
for any line segment $\gamma\colon [0,\ell]\to \Om$ with $\gamma(s):=(z,t+s)$ for some $t\in\R$ that
\[
|f(\gamma(0))-f(\gamma(\ell))|\le \liminf_{i\to\infty}\left(\int_{\gamma} g_i\,ds+2\omega(1/i)\right)
=\liminf_{i\to\infty}\int_{\gamma} g_i\,ds,
\]
and so
\[
\pV(f_z,\Om_z)
\le \liminf_{i\to\infty}\int_{\Om_z} g_i\,ds.
\]
We estimate
\begin{align*}
		&\int_{\pi_2(\Om)}\pV(f_z,\Om_z)\,d\mathcal L^{1}(z)\\
		&\qquad \le \int_{\pi_2(\Om)}\liminf_{i\to\infty}\int_{\Om_z} g_i\,ds\,d\mathcal L^{1}(z)\\
		&\qquad \le \liminf_{i\to\infty}\int_{\pi_2(\Om)}\int_{\Om_z} g_i\,ds\,d\mathcal L^{1}(z)
		\quad\textrm{by Fatou's lemma}\\
		&\qquad = \liminf_{i\to\infty}\int_{\Om} g_i\,d\mathcal L^{2}\quad\textrm{by Fubini}\\
		&\qquad \le   6400 \pi \kappa |f(\Om)| +12800  \pi \kappa^{-1}\int_{\Om} (h+2\eps) \,d\mathcal L^2
		+19200 \pi \kappa^{-1}  \eps
		\quad\textrm{by }\eqref{eq:gj estimate with sum}.
	\end{align*}
Recall \eqref{eq:pointwise variation in a direction}. Since we can do the above calculation
also in the $x_1$-coordinate direction, we obtain
\[
\Var(f,\Om) \le 12800\pi\kappa |f(\Om)| 
+25600\pi \kappa^{-1}\int_{\Om} (h+2\eps) \,d\mathcal L^2+38400  \pi\kappa^{-1} \eps.
\]
Letting $\eps\to 0$,
we get the estimate
\[
\Var(f,\Om)\le 25600\pi \left(\kappa |f(\Om)| +\kappa^{-1}\int_{\Om}h\,d\mathcal L^2\right)<\infty.
\]
All of the reasoning so far can be done also in every open subset $W\subset\Om$,
and so we have in fact
\[
|Df|(W)\le 25600\pi \left(\kappa |f(W)| +\kappa^{-1}\int_{W}h\,d\mathcal L^2\right).
\]
By considering small $\kappa>0$, we find that $|Df|$
is absolutely continuous with respect to $\mathcal L^2$ in $\Om$.
Thus we get $f\in D^1(\Om;\R^2)$, and choosing $\kappa=1$, we get the estimate
\[
\int_{\Om}|\nabla f|\,d\mathcal L^2
\le 25600\pi\left( |f(\Om)| +\int_{\Om} h \,d\mathcal L^2\right).
\]

Now we remove the assumption that $h$ is lower semicontinuous.
Using the Vitali-Carath\'eodory theorem (Theorem \ref{thm:VitaliCar})
we find a sequence $\{h_i\}_{i=1}^{\infty}$
of lower semicontinuous functions
in $L^1(\Om)$ such that $h\le h_{i+1}\le h_i$ for all $i\in\N$, and
$h_i\to h$ in $L^1(\Om)$.
Thus we get
\begin{equation}\label{eq:L1 energy}
\int_{\Om}|\nabla f|\,d\mathcal L^2
\le \liminf_{i\to\infty}25600\pi\left( |f(\Om)| +\int_{\Om} h_i \,d\mathcal L^2\right)
\le 25600\pi\left( |f(\Om)| +\int_{\Om} h \,d\mathcal L^2\right).
\end{equation}

Now we prove that in fact $f\in D^p(\Om;\R^2)$.
Note that $K_f^{\mathrm{fine}}<\infty$ a.e. in $\Om$.
Since $f\in D^1(\Om,\Om')\subset W_{\loc}^{1,1}(\Om;\Om')$, by Proposition \ref{prop:homeo} we have 
\begin{equation}\label{eq:nabla f p estimate}
K_f^{\mathrm{fine}}(x)^{1/2}J_f(x)^{1/2}
=\frac{2}{\pi^{1/2}}\Vert \nabla f(x)\Vert
\quad \textrm{for a.e. }x\in \Om.
\end{equation}
In the case $1<p<2$, by Young's inequality, and recalling that $h\ge K_f^{\mathrm{fine}}$, we get
\begin{equation}\label{eq:Lp energy}
\frac{2^p}{\pi^{p/2}}\int_{\Om} \Vert \nabla f\Vert ^p\,d\mathcal L^2
\le \int_{\Om} J_f\,d\mathcal L^2 + \int_{\Om} h^{p/(2-p)}\,d\mathcal L^2<\infty
\end{equation}
by \eqref{eq:Jacobian basic} and by the assumption $h\in L^{p^*/2}(\Om)$.
Thus $f\in D^p(\Om;\R^2)$.

In the case $p=2$, 
we have $K_f^{\mathrm{fine}}\le h\in L^{\infty}(\Om)$,
and then from \eqref{eq:nabla f p estimate} we get
\begin{equation}\label{eq:p is n case estimate}
	\Vert \nabla f(x)\Vert ^2
	\le \frac{\pi}{4}\Vert h\Vert_{L^{\infty}(\Om)}J_f(x)
	\quad \textrm{for a.e. }x\in \Om.
\end{equation}
This shows that $f\in D^{2}(\Om;\Om')$.
By Definition \ref{def:quasiconformal} and the discussion below it, we then have that in fact
$f$ is quasiconformal;
note that $\nabla f(x)$ is now a classical gradient for a.e. $x\in\Om$.
Moreover, using e.g. \cite[Theorem 9.8]{HKST} we know that $f^{-1}$ is also quasiconformal and thus absolutely
continuous in measure,
and so $J_f(x)>0$ for a.e. $x\in \Om$. Thus by Proposition \ref{prop:homeo}, we have
$K_f(x)=K_f^{\mathrm{fine}}(x)$ for a.e. $x\in\Om$, and
from \eqref{eq:p is n case estimate} we obtain that \eqref{eq:Ln energy thm} holds.

Finally, using the Dirichlet energy estimates \eqref{eq:L1 energy}, \eqref{eq:Lp energy},
and \eqref{eq:p is n case estimate},
it is easy to generalize to the case where $\Om$ is unbounded.
\end{proof}

\begin{proof}[Proof of Theorem \ref{thm:main}]
	For every direction $v\in \partial B(0,1)$, the
	intersection of $E$ with almost every line $L$ parallel to $v$
	is at most countable, see e.g. \cite[p. 103]{Vai}.
	For every bounded open set $\Om\subset \R^2$,
	we have $|f(\Om)|<\infty$ since $f$ is a homeomorphism, and then by Theorem \ref{thm:main text}
	we have $f\in D^p(\Om;\R^2)$,
	and in the case $p=2$, moreover $K_f(x)=K_f^{\mathrm{fine}}(x)$ for a.e. $x\in\Om$ and
	\[
	\Vert \nabla f(x)\Vert ^2
	\le \frac{\pi}{2}\Vert K_f^{\mathrm{fine}}\Vert_{L^{\infty}(\Om)}J_f(x)
	\quad \textrm{for a.e. }x\in \Om.
	\]
	Thus $f\in W^{1,p}_{\loc}(\R^2;\R^2)$,
	and in the case $p=2$, we have $K_f(x)=K_f^{\mathrm{fine}}(x)$ for a.e. $x\in\R^2$ and
	\[
	\Vert \nabla f(x)\Vert ^2
	\le \frac{\pi}{2}\Vert K_f^{\mathrm{fine}}\Vert_{L^{\infty}(\R^n)}J_f(x)
	\quad \textrm{for a.e. }x\in \R^2.
	\]
	Thus $f$ is quasiconformal in $\R^2$.
\end{proof}

\begin{proof}[Proof of Corollary \ref{cor:main}]
This follows immediately from Theorem \ref{thm:main}.
\end{proof}

In closing, we note that certain open problems arise naturally from our work.
Koskela--Rogovin \cite[Corollary 1.3]{KoRo} prove that in the definition of $K_f$,
one can replace ``$\limsup$'' by ``$\liminf$'', and the result
(analogous to Theorem \ref{thm:main}) still holds.
Thus one can ask: does Theorem \ref{thm:main} still hold if
``$\limsup$'' is replaced by ``$\liminf$'' in the definition of $K_f^{\mathrm{fine}}$?

Another question is: can our results be generalized to $\R^n$ with $n\ge 3$,
and further to metric measure spaces?
Much of the literature on the topic, discussed in the Introduction, in fact deals with the setting of quite general
metric measure spaces.
We observe that most of the quantities and techniques used in the proof of Theorem \ref{thm:main}
make sense also in metric spaces, but in deriving
\eqref{eq:big connected set} we relied heavily
on the structure of the plane.

\end{document}